\documentclass[12pt,draft]{amsart}
\usepackage[all]{xy}
\usepackage{amsfonts}
\usepackage{amssymb}

\usepackage{enumerate}

\usepackage{color}

\renewcommand{\le}{\leqslant}
\renewcommand{\ge}{\geqslant}

\newtheorem{teo}{Theorem}[section]
\newtheorem{lem}[teo]{Lemma}
\newtheorem{prop}[teo]{Proposition}

\newtheorem{cor}[teo]{Corollary}

\theoremstyle{definition}
\newtheorem{dfn}[teo]{Definition}
\newtheorem{rk}[teo]{Remark}
\newtheorem{ex}[teo]{Example}

\def\<{\langle}
\def\>{\rangle}

\def\r{\rho}

\def\f{{\varphi}}

\def\C{{\mathbb C}}
\def\R{{\mathbb R}}
\def\Z{{\mathbb Z}}

\def\Ker{\mathop{\rm Ker}\nolimits}

\def\Id{\operatorname{Id}}

\def\Tr{\operatorname{Tr}}

\def\Spec{\operatorname{Spectr}}

\def\ind{\operatorname{ind}}

\def\1{\mathbf 1}

\newcommand{\coker}{{\rm Coker\ }}
\newcommand{\im}{{\rm Im\ }}

\newcommand{\wh}[1]{\widehat{#1}}

\def\Fix{\operatorname{Fix}}

\def\cR{{\mathcal R}}
\def\N{{\mathbb N}}

\def\divides{{\mathchoice{\mathrel{\bigm|}}{\mathrel{\bigm|}}{\mathrel{|}}{\mathrel{|}}}}
\def\notdivides{\mathrel{\kern-3pt\not\!\kern3.5pt\bigm|}}
\def\smallnotdivides{\mathrel{\kern-2pt\not\!\kern3.5pt\vert}}
\def\Divides{\divides\!\divides}

\begin{document}

\title[Dynamical zeta functions and representations spaces]
{Dynamical zeta functions of Reidemeister type and representations spaces}

\author{Alexander Fel'shtyn}
\thanks{The work 
is funded by the  Narodowe Centrum Nauki of Poland (NCN) (grant No.~\!2016/23/G/ST1/04280 (Beethoven 2)).}
\address{Instytut Matematyki, Uniwersytet Szczecinski,
ul. Wielkopolska 15, 70-451 Szczecin, Poland} \email{fels@wmf.univ.szczecin.pl, felshtyn@gmail.com}
\author{Malwina Zietek}
\address{Instytut Matematyki, Uniwersytet Szczecinski,
ul. Wielkopolska 15, 70-451 Szczecin, Poland} \email{malwina.zietek@gmail.com}

\dedicatory{Dedicated to the memory of Sergiy  Kolyada}

\subjclass[2000]{Primary 37C25; 37C30; 22D10;  Secondary 20E45; 54H20; 55M20}
\date{XXX 1, 2014 and, in revised form, YYY 22, 2015.}

\keywords{ Twisted conjugacy class, Reidemeister number, Reidemeister zeta function,  unitary dual}

\begin{abstract}

In this paper we continue to study the Reidemeister zeta function.
We prove P\'olya -- Carlson dichotomy between rationality and a natural boundary for analytic behavior  of the Reidemeister zeta function for a  large class of automorphisms of  Abelian groups.
 We also study dynamical representation theory zeta functions counting
numbers of fixed  irreducible representations for iterations of an endomorphism.
The rationality and functional equation for these zeta functions are proven for several classes of groups. We  find a connection between these zeta
functions and the Reidemeister torsions of the corresponding mapping tori.
We also establish   the connection between the Reidemeister  zeta function and  dynamical representation theory zeta functions  under restriction of endomorphism to a subgroup and to a quotient group.

\end{abstract}

\maketitle
\section{Introduction}

Let $G$ be a countable discrete group and $\phi: G\rightarrow G$ an endomorphism.
Two elements $\alpha,\beta\in G$ are said to be
$\phi$-{\em conjugate} or {\em twisted conjugate,} iff there exists $g \in G$ with
$\beta=g  \alpha  \phi(g^{-1}).$
We shall write $\{x\}_\phi$ for the $\phi$-{\em conjugacy} or
{\em twisted conjugacy} class
of the element $x\in G$.
The number of $\phi$-conjugacy classes is called the {\em Reidemeister number}
of an  endomorphism $\phi$ and is  denoted by $R(\phi)$.
If $\phi$ is the identity map then the $\phi$-conjugacy classes are the usual
conjugacy classes in the group $G$.
Taking a dynamical point of view, we consider the iterates of  $\phi$, and we may define \cite{Fel91}   a Reidemeister zeta function of  $\phi$ as a power series:
\begin{align*}
R_\phi(z)&=\exp\left(\sum_{n=1}^\infty \frac{R(\phi^n)}{n}z^n\right).
\end{align*}
Whenever we mention the Reidemeister zeta function $R_\phi(z)$, we shall assume that it is well-defined and so  $R(\phi^n)<\infty$ for all $n>0$.
The following problem was investigated \cite{fh}:  for which groups and endomorphisms is the Reidemeister zeta function a rational function? Is this zeta function an algebraic function?

When a Reidemeister zeta function is a rational function  the infinite sequence of coefficients of the corresponding power series is closely interconnected, and is given by the finite set of zeros and poles of the zeta function.
In \cite{Fel91, fh, Li, fhw, Fel00}, the rationality of the Reidemeister zeta function $R_\phi(z)$ was proven in the following cases: the group is finitely generated and an endomorphism is eventually commutative; the group is finite; the group is a direct sum of a finite group and a finitely generated free Abelian group; the group is finitely generated, nilpotent and torsion free.  Recently, the rationality and functional equation for the Reidemeister zeta function were  proven for endomorphisms of fundamental groups of infra-nilmanifolds \cite{DeDu}
and for endomorphisms of fundamental groups of infra-solvmanifolds of type (R) \cite{FelLee}.

In this paper we continue to study the Reidemeister zeta function.
We prove P\'olya -- Carlson dichotomy between rationality and a natural boundary for analytic behavior  of Reidemeister zeta function for a  large class of automorphisms of  Abelian groups.

 We continue to study dynamical representation theory zeta functions (see \cite{FTZ}) counting
numbers of fixed  irreducible unitary representations for iterates of an endomorphism.
The rationality and functional equation for these zeta functions are proven for several classes of groups.
We find a connection between these  zeta functions and the Reidemeister torsions of the corresponding  mapping tori.

We establish the connection
  between  Reidemeister zeta function and  dynamical representation theory zeta functions  under restriction of endomorphism to a subgroup and to a quotient group. 

Our method is to identify the Reidemeister numbers with  the number of fixed points of the induced
map $\hat\phi$ (respectively, its iterations)
of an appropriate subspace of the unitary dual $\wh G$,
when $R(\f)<\infty$.
This method is called the twisted Burnside--Frobenius
theory (TBFT), because in the case of a finite group and
identity automorphism we arrive to the classical Burnside--Frobenius
theorem on enumerating of (usual) conjugacy classes via irreducible unitary representations.

Let us present the contents of the paper in more details.

In Section 2  the rationality and functional equation for 
dynamical representation theory zeta functions are proven for 
endomorphisms of finitely generated Abelian  groups; endomorphisms of finitely generated torsion free nilpotent groups; endomorphisms of  groups with finite $\phi$-irreducible subspaces of corresponding unitary dual spaces and for automorphisms of crystallographic groups with diagonal holonomy $Z_2$.
For a periodic  automorphism  of a group we have proved a product formula for dynamical representation theory zeta functions which implies that these zeta functions are radicals of rational functions.

In Section 3 we investigate the rationality of these zeta functions  and the connections  between  Reidemeister zeta function and  dynamical representation theory zeta functions  under restriction of endomorphism to a subgroup and to a quotient group.
We also  prove  the Gauss congruences for the Reidemeister numbers of iterations of endomorphism for  a 
group with polycyclic quotient group.

In Section 4 we presents results in support of a  P\'olya -- Carlson dichotomy between rationality and a natural boundary for analytic behavior  of Reidemeister zeta function for a  large class of automorphisms of  Abelian groups.

\noindent
\textbf{Acknowledgments.}
This work was supported by the grant Beethoven 2 of the Narodowe Centrum Nauki of Poland(NCN),
grant No. 2016/23/G/ST1/04280.

%The authors are grateful to Evgenij Troitsky  and Tom Ward for helpful comments and
%valuable discussions. 

\section{Dynamical zeta functions and representations spaces}
Suppose, $\phi$ is an endomorphism of a discrete group $G$.
\rm
Denote by $\wh G$ the \emph{unitary dual} of $G$, i.e.
the space of equivalence classes of
unitary irreducible representations of $G$, equipped with the
\emph{hull-kernel} topology, denote by
$\wh G_f$ the subspace of the unitary dual formed by
irreducible finite-dimensional representations,
and by $\wh G_{ff}$ the subspace of $\wh G_{f}$ formed by
 \emph{finite} representations, i.e. representations
 that factorize through a finite group.
Generally the correspondence $\widehat{\phi}:\rho\mapsto \rho\circ\phi$
does not define a dynamical system (an action of the semigroup
of positive integers) on the unitary dual $\widehat{G}$    
or its finite-dimensional part $\widehat{G}_f$, or finite part
$\widehat{G}_{ff}$, because in contrast with the
authomorphism case, the representation $\rho\circ\phi$
may be reducible, so it is only possible to decompose $\rho\circ\phi$ into irreducible
components and we obtain a sort of multivalued map $\widehat{\phi}$.

Nevertheless we can consider representations $\rho$ such that
$\rho \sim \rho\circ\phi$ and proceed as follows.

\begin{dfn}\label{dfn:represreidnum}
{\rm
A \emph{representation theory Reidemeister number}
$RT(\phi)$
is defined \cite{FTZ} as  the number of all $[\rho]\in \widehat{G}$ such that
$\rho \sim \rho\circ\phi$. Taking $[\rho]\in \widehat{G}_f$
(respectively $[\rho]\in \widehat{G}_{ff}$) we obtain
$RT^f(\phi)$ (respectively $RT^{ff}(\phi)$). Evidently
$RT(\phi)\ge RT^f(\phi)\ge RT^{ff}(\phi)$.
}
\end{dfn}

In analogy with the Reidemeister zeta function and other similar objects we have defined  in \cite{FTZ}  jointly with E.Troitsky
 following dynamical  representation zeta functions 
\begin{align*}
RT_\phi(z)&=\exp\left(\sum_{n=1}^\infty \frac{RT(\phi^n)}{n}z^n\right),\\
RT^f_\phi(z)&=\exp\left(\sum_{n=1}^\infty \frac{RT^f(\phi^n)}{n}z^n\right),\\
RT^{ff}_\phi(z)&=\exp\left(\sum_{n=1}^\infty \frac{RT^{ff}(\phi^n)}{n}z^n\right),\\
\end{align*}
when  numbers $ RT(\phi^n)$
(resp, $ RT^f(\phi^n)$  , or $ RT^{ff}(\phi^n)$) are all finite.

The importance of these numbers is justified by the
following dynamical interpretation.
In \cite{RJMP} the following
``dynamical part'' of the dual space, where $\widehat{\phi}$ and all
its iterations $\widehat{\phi}^n$  define a dynamical system,
was defined. 

\begin{dfn}
{\rm
Following \cite{RJMP} a class $[\rho]$ is called
a $\wh\phi$-\textbf{f}-point,
if $\rho\sim \rho\circ\phi$ (so, these are the points
under consideration in the  Definition \ref{dfn:represreidnum}).
}
\end{dfn}

\begin{dfn}
{\rm
Following \cite{RJMP} an element $[\rho]\in \wh{G}$ (respectively, in $\wh{G}_f$
or $\wh{G}_{ff}$)
is called $\phi$-\emph{irreducible} if $\r\circ \phi^n$ is
irreducible for any $n=0,1,2,\dots$.

Denote the corresponding subspaces of $\wh{G}$ (resp., $\wh{G}_f$
or $\wh{G}_{ff}$)
by $\wh{G}^\phi$ (resp., $\wh{G}^\phi_f$ or $\wh{G}^\phi_{ff}$). 
}
\end{dfn}

\begin{lem}[Lemma 2.4 in \cite{RJMP}]\label{lem:keylem}
Suppose the representations $\rho$ and $\rho\circ\phi^n$ are
equivalent  for some $n\geq 1$. Then $[\rho]\in \wh G^\phi$.
\end{lem}

\begin{cor}[Corollary 2.5 in \cite{RJMP}]\label{cor:periodicanddyn}
Generally, there is no dynamical system defined by $\widehat{\phi}$
on $\wh G$ (resp., $\wh G_f$, or $\wh{G}_{ff}$).
We have only the well-defined notion of a
$\wh\phi^n$-\textbf{f}-point.

A well-defined dynamical system exists on 
$\wh G^\phi$ (resp., $\wh G^\phi_f$, or $\wh{G}^\phi_{ff}$). 
Its $n$-periodic points are exactly $\wh\phi^n$-\textbf{f}-points.
\end{cor}

We refer to \cite{RJMP} for proofs and
details.

Once we have identified
the coefficients of representation theory zeta functions 
with the
numbers of periodic points of a dynamical system, 
the standard argument with the
M\"obius inversion formula
(see e.g. \cite[p.~104]{Fel00}, \cite{RJMP}) 
gives the following statement.

\begin{teo} (Theorem 2.7 of \cite{FTZ})\label{teo:congrue_rep_reide}
Suppose , $RT(\phi^n)<\infty$ for any $n$.
Then we have the following Gauss congruences
for representation theory Reidemeister numbers:
 $$
 \sum_{d\mid n} \mu(d)\cdot RT(\phi^{n/d}) \equiv 0 \mod n
 $$
for any $n$.

A similar statement is true for $RT^f(\phi^n)$ and 
$RT^{ff}(\phi^n)$.
\end{teo}

Here the above \emph{M\"obius function} is defined as
$$
\mu(d) =
\left\{
\begin{array}{ll}
1 & {\rm if}\ d=1,  \\
(-1)^k & {\rm if}\ d\ {\rm is\ a\ product\ of}\ k\ {\rm distinct\ primes,}\\
0 & {\rm if}\ d\ {\rm is\ not\ square-free.}
\end{array}
\right.
$$

\begin{dfn}\label{dfn:tbft}
Following \cite{RJMP} we say that TBFT (resp., TBF$T_f$, TBF$T_{ff}$) takes place for
an endomorphism $\phi:G\to G$ and its iterations, if
$R(\phi^n)<\infty$ and $R(\phi^n)$ coincides with the
number of $\wh\phi^n$-\textbf{f}-points in $\wh G$
(resp., in $\wh G_f$, $\wh G_{ff}$)
for all $n\in \mathbb{N}$.

Similarly, one can give a definition for a single
endomorphism (without iterations).
\end{dfn}

The following statement  follows from the 
definitions.

\begin{prop}(Proposition 2.8 of \cite{FTZ}) \label{prop:tbftimplcoin}
Suppose, $\phi:G\to G$ is an endomorphism and $R(\phi)<\infty$.
If TBFT $($resp., TBFT$_f)$ is true for $G$ and $\phi$,
then $R(\phi)=RT(\phi)$ $($resp., $R(\phi)=RT^f(\phi)=RT^{ff}(\phi))$.

If the suppositions hold for $\phi^n$, for any $n$, then
$R_\phi(z)=RT_\phi(z)$ $($resp., 
$R_\phi(z)=RT^f_\phi(z)=RT^{ff}_\phi(z))$.
\end{prop}

Denote  by  $ AM^f(\phi^n)$   the number of \emph{isolated} $n$-periodic
points(i.e. \emph{isolated}  $\wh\phi^n$-\textbf{f}-points) of the  dynamical system $(\wh{\phi})^n$  on 
 $\wh G^\phi_f$.

If these numbers are finite for all powers of $\phi$, the corresponding
Artin--Masur representation zeta function is defined as
$$
AM^f_\phi(z) =\exp\left(\sum_{n=1}^\infty \frac{AM^f(\phi^n)}{n}z^n\right).
$$

Let $Z(\phi)$ be one of the numbers $RT(\phi)$, $ RT^f(\phi)$, $ RT^{ff}(\phi)$, $ AM^f(\phi)$ . Let 
$$
Z_\phi(z)=\exp\left(\sum_{n=1}^\infty \frac{Z(\phi^n)}{n}z^n\right)
$$
 be one of the zeta
functions  $AM^f_\phi(z)$, $RT_\phi(z)$, $RT^f_\phi(z)$, $RT^{ff}_\phi(z).$

\begin{teo}\label{period}
Let $\phi$ be a  periodic automorphism of least period $m$ of  a group $G$ . Then the 
 zeta function $Z_\phi(z)$ is equal to 
$$
Z_\phi(z)=\prod_{d\mid m}\sqrt[d]{(1-z^d)^{-P(d)}},
$$
 where the product is taken over all divisors $d$ of the period $m$, and $P(d)$ is the integer
$$  P(d) = \sum_{d_1\mid d} \mu(d_1)Z({\phi}^{d\mid d_1}) .  $$
\end{teo}

\begin{proof}

Since $\phi^m = id $, then  $(\wh{\phi})^m =id$ as well and $ Z(\phi^j)=Z(\phi^{m+j})$ for every $j$. If $(k,m)=1$, there exist positive integers $t$
and $q$ such that $kt=mq+1$. So $ (\phi^k)^t=\phi^{kt}= \phi^{mq+1}=\phi^{mq}\phi=(\phi^m)^{q}\phi =  \phi$.
Consequently, $  Z(\phi^k)=Z(\phi) $.
The same argument shows that $Z(\phi^d)=Z(\phi^{di})$ if $(i,m/d)=1$ where $d$ divisor $m$ 
Using these series of equal numbers we obtain by direct calculation 
\begin{eqnarray*}
Z_\phi(z) & = & \exp\left(\sum_{i=1}^\infty \frac{Z(\phi^i)}{i} z^i \right)
          = \exp\left(\sum_{d\mid m} \sum _{i=1}^\infty \frac{P(d)}{d}\cdot\frac{{z^d}^i}{i}\right)\\
	   &=&\exp\left(\sum_{d\mid m}\frac{P(d)}{d}\cdot \log (1-z^d)\right)
	    =  \prod_{d\mid m}\sqrt[d]{(1-z^d)^{-P(d)}}
\end{eqnarray*}
where the integers $P(d)$ are calculated recursively by the formula
$$
P(d)= Z(\phi^d)  - \sum_{d_1\mid d; d_1\not=d} P(d_1).
$$
Moreover, if the last formula is rewritten in the form
$
Z(\phi^d)=\sum_{d_1\mid d} P(d_1)
$ 
and one uses  the M\"obius inversion law for real function in number theory, then
$$
P(d)=\sum_{d_1\mid d}\mu(d_1)\cdot Z(\phi^{d/d_1}),
$$
where $\mu(d_1)$ is the M\"obius function in number theory. The theorem is proved.
\end{proof}
\begin{cor}
If in Theorem \ref{period}  the period $m$ is a prime number, then
$$ 
Z_\phi(z) =   \frac{1}{(1-z)^{Z(\phi)}}\cdot \sqrt[m]{(1-z^m)^{Z(\phi) - Z(\phi^m)}}.
$$

\end{cor}

\begin{teo}\label{finite}

Let $\phi:G\rightarrow G$ be an endomorphism
of group $G$. Suppose that subspaces $\wh{G}^\phi$, $\wh{G}_f^\phi$, and $\wh{G}_{ff}^\phi$ are finite.
Then zeta function $Z_\phi(z)$ is a rational function satisfying a functional equation
$$
 Z_\phi\left(\frac{1}{z}\right)
 =
 (-1)^a z^b Z_\phi(z).
$$
In particular we have
\begin{equation}
 Z_\phi(z) = \prod_{[\gamma]} \frac{1}{1-z^{\#[\gamma]}},
\end{equation}
where the product is taken over the periodic orbits of the  dynamical system $(\wh{\phi})^n$ 
in $\wh{G}^\phi$, resp $\wh{G}_f^\phi$, or $\wh{G}_{ff}^\phi$.
In the functional equation the numbers $a$ and $b$ are respectively
the number of periodic $\wh{\phi}$-orbits of elements
of $\wh{G}^\phi$, resp $\wh{G}_f^\phi$, or $\wh{G}_{ff}^\phi$ and the number of periodic elements
of $\wh{G}^\phi$, resp $\wh{G}_f^\phi$, or $\wh{G}_{ff}^\phi$.
\end{teo}

\begin{proof}
We shall call an element of $\wh{G}^\phi$, resp $\wh{G}_f^\phi$, or $\wh{G}_{ff}^\phi$ periodic if it is fixed by some
iteration of $\wh{\phi}$.
A periodic element $\gamma$ is fixed by $\wh{\phi}^n$ iff $n$ is divisible by
the cardinality the orbit of $\gamma$.
We therefore have
\begin{eqnarray*}
 Z(\phi^n)
  =  
 \sum_{\gamma \ periodic \atop \#[\gamma]\mid n} 1 
  =  
 \sum_{[\gamma]\ such \ that, \atop \#[\gamma]\mid n} \#[\gamma].
\end{eqnarray*}

From this follows
\begin{eqnarray*}
 Z_\phi(z)
  = 
 \exp\left(\sum_{n=1}^\infty \frac{Z(\phi^n)}{n} z^n\right) 
  = 
 \exp\left(\sum_{[\gamma]}
           \sum_{n=1\atop \#[\gamma]\mid n}^\infty
           \frac{\#[\gamma]}{n} z^n\right) \\
  = 
 \prod_{[\gamma]}
 \exp\left(\sum_{n=1}^\infty
 \frac{\#[\gamma]}{\#[\gamma]n} z^{\#[\gamma]n}\right) 
  = 
 \prod_{[\gamma]}
 \exp\left(\sum_{n=1}^\infty
 \frac{1}{n} z^{\#[\gamma]n}\right) \\
  = 
 \prod_{[\gamma]}
 \exp \left( - \log \left( 1-z^{\#[\gamma]}\right)\right) 
  = 
 \prod_{[\gamma]} \frac{1}{1-z^{\#[\gamma]}}.
\end{eqnarray*}

Moreover
\begin{equation*}
\begin{aligned}
Z_\phi\left(\frac{1}{z}\right) {}&= \prod_{[\gamma]}\frac{1}{1-z^{-\#[\gamma]}} 
= \prod_{[\gamma]}\frac{z^{\#[\gamma]}}{z^{\#[\gamma]}-1}
= \prod_{[\gamma]}\frac{-z^{\#[\gamma]}}{1-z^{\#[\gamma]}}
\\&= \prod_{[\gamma]} -z^{\#[\gamma]}Z_{\phi}(z)
= (-1)^{\#\{[\gamma]\}}z^{\sum\#[\gamma]}Z_{\phi}(z).
\end{aligned}
\end{equation*}
\end{proof}
\subsection{Endomorphisms of finitely generated Abelian groups}

For a finitely generated Abelian group $G$ we define the finite subgroup $G^{finite}$ to be the subgroup of
torsion elements of $G$. We denote the quotient $G^\infty:=G/G^{finite}$.
The group $G^\infty$ is torsion free.
Since the image of any torsion element by a homomorphism must be a torsion
element, the function $\phi:G\to G$ induces maps
 $$
 \phi^{finite}:G^{finite}\longrightarrow G^{finite},\;\;\;\;
 \phi^\infty:G^\infty\longrightarrow G^\infty.
 $$

 If $G$ is abelian, then  $\wh{G}=\wh{G}_f=\wh{G}^\f=\wh{G}^\f_f$ \cite{RJMP}.
 
The Lefschetz zeta function of a discrete dynamical system $\hat{\phi}$ equals:
   $$
  L_{\hat{\phi}}(z) := \exp\left(\sum_{n=1}^\infty \frac{L(\hat{\phi}^n)}{n} z^n \right),
 $$
  where
 $$
   L(\hat{\phi}^n) := \sum_{k=0}^{\dim X} (-1)^k \Tr\Big[\hat{\phi_{*k}}^n:H_k(\hat{G};Q)\to H_k(\hat{G};Q)\Big]
 $$
 is the Lefschetz number of $\hat{\phi}^n$. 
 The Lefschetz zeta function $ L_{\hat{\phi}}(z)$ is a rational function of $z$ and
is given by the formula:
$$
 L_{\hat{\phi}}(z) = \prod_{k=0}^{\dim X}
          \det\big(I-\hat{\phi_{*k}}\cdot z\big)^{(-1)^{k+1}}.
$$

\begin{teo}\label{fingenabelian}
Let $\phi:G\to G$ be an endomorphism of a finitely
generated Abelian group.
Then we have
\begin{equation}
  Z(\phi^n) = \mid L(\hat{\phi}^{n}) \mid,
\end{equation}
where  $L(\hat{\phi}^{n})$ is the Lefschetz number
of $\hat{\phi}$ thought of as a self-map of the topological space $\hat{G}$.
{}From this it follows that zeta functon $Z_\phi(z)$ is a rational function  and is equal to:
\begin{equation}
  Z_\phi(z) = L_{\hat{\phi}}(\sigma z)^{(-1)^r},
\end{equation}
where  $\sigma=(-1)^p$ where $p$ is the number of real eingevalues
$\lambda\in \Spec( \phi^\infty)$ such that $\lambda<-1$
and $r$ is the number of  real eingevalues
$\lambda\in \Spec (\phi^\infty)$
such that $\mid\lambda\mid > 1$.
If $G$ is a finite abelian group  then this reduces to
$$  Z(\phi^n)=L(\hat{\phi}^n)  \; {\rm and} \; Z_\phi(z)=L_{\hat{\phi}}(z).$$
\end{teo}
\begin{proof}

If $G$ is finite abelian  then $\hat{G}$ is a discrete
finite set, so the number of fixed points is equal to
the Lefschetz number. This finishes the proof in the
case that $G$ is finite.

If $G$ is a finitely generated free Abelian group then
 the dual of $G$
is  a torus whose dimension is equal to the rank of $G$.
The dual of any finitely generated
discrete Abelian group is the direct sum of a torus and a
finite group.
 
If $G$ a finitely
generated Abelian group it is only necessary to check that
the number of fixed points of $\hat{\phi}^n$ is equal to the
absolute value of its Lefschetz number.
We assume without loss of generality that $n=1$.
We are assuming that $Z(\phi)$ is finite, so
the fixed points of $\hat{\phi}$ form a discrete set.
We therefore have
 $$
 L(\hat{\phi})
 =
 \sum_{x\in\Fix\hat{\phi}} \ind(\hat{\phi},x).
 $$
Since $\phi$ is a group endomorphism, the zero element $0\in\hat{G}$ is always fixed.
Let $x$ be any fixed point of $\hat{\phi} $.
We then have a commutative diagram
 $$
 \begin{array}{ccccc}
 g & \hat{G} & \stackrel{\hat{\phi}}{\longrightarrow} & \hat{G} & g \\
 \updownarrow &  \updownarrow & & \updownarrow & \updownarrow \\
 g + x & \hat{G} & \stackrel{\hat{\phi}}{\longrightarrow} & \hat{G} & g + x
 \end{array}
 $$
in which the vertical functions are translations on $\hat{G}$ by $x$.
Since the vertical maps map $0$ to $x$, we deduce that
 $$ \ind(\hat{\phi},x) = \ind(\hat{\phi},0) $$
and so all fixed points have the same index.
It is now sufficient to show that $\ind(\hat{\phi},0)=\pm 1$.
This follows because the map on the torus
 $$ \hat{\phi}:\hat{G}_0\to\hat{G}_0 $$
lifts to a linear map of the universal cover, which is in this case the
Lie algebra of $\hat{G}$. The index is then the sign of the determinant of the identity map  minus this lifted map.
This determinant cannot be zero, because $1-\hat{\phi}$ must have finite
kernel by our assumption that the $Z(\phi)$ is
finite
(if $\det(1-\hat{\phi})=0$ then the kernel of $1-\hat{\phi}$
is a positive dimensional subgroup of $\hat{G}$, and therefore
infinite). So we have  $Z(\f^n) = \mid L(\hat{\phi}^{n}) \mid =(-1)^{r+pn}L(\hat{\phi}^{n})$ for all $n$ (see also \cite{Fel00}).

Then the zeta function  $$ Z_\phi(z) = L_{\hat{\phi}}(\sigma z)^{(-1)^r}$$ is rational function as well.
\end{proof}

\subsubsection{ Functional equation }
To write down a functional equation for the Reidemeister type zeta functions, we recall the following functional equation for the Lefschetz zeta function:
\begin{lem}[{\cite[Proposition~8]{fri1}}, see also \cite{del}] \label{Fried}
{Let $M$ be a closed orientable manifold of dimension $m$ and let $f:M\to M$ be a continuous map of degree $d$. Then
$$
L_{f}\left(\frac{\alpha}{dz}\right)=\epsilon\,(-\alpha dz)^{(-1)^m\chi(M)}\,L_{f}(\alpha z)^{(-1)^m}
$$
where $\alpha=\pm1$ and $\epsilon\in\C$ is a non-zero constant such that if $|d|=1$ then $\epsilon=\pm1$.}
\end{lem}

We obtain:

\begin{teo}[{Functional Equation}]\label{FE}
Let $\phi:G\to G$ be an endomorphism of a finitely
generated Abelian group of the rank $\geq 1$.
 Then the  zeta function $Z_{\phi}(z)$, whenever it is defined, has the following functional equation:
\begin{equation*}
Z_{\phi}\left(\frac{1}{dz}\right)
= Z_{\phi}(z)^{(-1)^m}\epsilon^{(-1)^{r}}
\end{equation*}
where $d$ is a  degree $\hat\phi$,  $m= \dim \hat G$, $\epsilon$ is a constant in $\C^\times$, $\sigma=(-1)^r$, $p$ is the number of real eigenvalues of $\phi^\infty $ which are $>1$ and  $r$ is the number of  real eingevalues
$\lambda\in \Spec (\phi^\infty)$
such that $\mid\lambda\mid > 1$.
If $|d|=1$ then $\epsilon=\pm1$.
\end{teo}

\begin{proof}
 We have  $ Z_{\phi}(z) = L_{\hat\phi}(\sigma z)^{(-1)^{r}}$. By Lemma~\ref{Fried}
\begin{align*}
Z_{\phi}\left(\frac{1}{dz}\right) = L_{\hat\phi}\left(\frac{\sigma}{dz}\right)^{(-1)^{r}}
 =\left(\epsilon(-\sigma dz)^{(-1)^m\chi(\hat G)}L_{\hat\phi}(\sigma z)^{(-1)^m}\right)^{(-1)^{r}}= \\
= Z_{\phi}(z)^{(-1)^m}\epsilon^{(-1)^{r}}(-\sigma dz)^{(-1)^{m+r}\chi(\hat G)}.
\end{align*}
On the other hand  $\chi(\hat G)=0$ because the dual $\hat G$ of any finitely generated
discrete Abelian group  of the rank $\geq 1$ is the direct sum of a torus of $\dim \geq 1$ and a
finite group, i.e. $\hat G$ is a union of finitely many tori. This finishes our proof.
\end{proof}

\subsection{Endomorphisms of nilpotent  groups and crystallographic groups}

\begin{teo} \label{teo:nil}
Let $\phi:G\to G$ be an endomorphism of a finitely generated  torsion free nilpotent  group $G$ or 
let $\phi$ be  an automorphism of crystallographic group $G$ with diagonal holonomy $Z_2$.
Then the zeta function $RT^f_\phi(z)=RT^{ff}_\phi(z))$ is  rational function.
\end{teo}
\begin{proof}
Any  finitely generated  torsion free nilpotent  group is a  supersolvable, hence, polycyclic group.
Any crystallographic group with diagonal holonomy $Z_2$ is 
a polycyclic-by-finite group. In \cite{RJMP, crelle} twisted Burnside-Frobenius theorem($TBFT_f$ and $TBFT_{ff}$) was proven for endomorphisms of polycyclic groups  and for automorphisms of polycyclic-by-finite groups. This theorem implies equality of Reidemeister zeta function  $R_\phi(z)$ and zeta function $RT^f_\phi(z)=RT^{ff}_\phi(z))$. In \cite{Fel00} the rationality of the Reidemeister zeta function $R_\phi(z)$
was proven for endomorphisms of a finitely generated  torsion free nilpotent  groups and in
\cite{DekTerBus} the rationality of $R_\phi(z)$ was proven for  automorphisms of crystallographic groups with diagonal holonomy $Z_2$. This completes the proof.

\end{proof}

\section{Connection with Reidemeister Torsion}

Like the Euler characteristic, the Reidemeister torsion is algebraically defined.
 Roughly speaking, the Euler characteristic is a
graded version of the dimension, extending
the dimension from a single vector space to a complex
of vector spaces.
In a similar way, the Reidemeister torsion
is a graded version of the absolute value of the determinant
of an isomorphism of vector spaces.
Let $d^i:C^i\rightarrow C^{i+1}$ be a cochain complex $C^*$
of finite dimensional vector spaces over $\C$ with
$C^i=0$ for $i<0$ and large $i$.
If the cohomology $H^i=0$ for all $i$ we say that
$C^*$ is {\it acyclic}.
If one is given positive densities $\Delta_i$ on $C^i$
then the Reidemeister torsion $\tau(C^*,\Delta_i)\in(0,\infty)$
for acyclic $C^*$ is defined as follows:
 
\begin{dfn}
 Consider a chain contraction $\delta^i:C^i\rightarrow C^{i-1}$,
 ie. a linear map such that $d\circ\delta + \delta\circ d = id$.
 Then $d+\delta$ determines a map
 $ (d+\delta)_+ : C^+:=\oplus C^{2i}
                      \rightarrow C^- :=\oplus C^{2i+1}$
and a map
$ (d+\delta)_- : C^- \rightarrow C^+ $.
Since the map
$(d+\delta)^2 = id + \delta^2$ is unipotent,
$(d+\delta)_+$ must be an isomorphism.
One defines $\tau(C^*,\Delta_i):= \mid \det(d+\delta)_+\mid$
(see \cite{fri2}).
\end{dfn}

Reidemeister torsion is defined in the following geometric setting.
Suppose $K$ is a finite complex and $E$ is a flat, finite dimensional,
complex vector bundle with base $K$.
We recall that a flat vector bundle over $K$ is essentially the
same thing as a representation of $\pi_1(K)$ when $K$ is
connected.
If $p\in K$ is a basepoint then one may move the fibre at $p$
in a locally constant way around a loop in $K$. This
defines an action of $\pi_1(K)$ on the fibre $E_p$ of $E$
above $p$. We call this action the holonomy representation
$\rho:\pi\to GL(E_p)$.

Conversely, given a representation $\rho:\pi\to GL(V)$
of $\pi$ on a finite dimensional complex vector space $V$,
one may define a bundle $E=E_\rho=(\tilde{K}\times V) / \pi$.
Here $\tilde{K}$ is the universal cover of $K$, and
$\pi$ acts on $\tilde{K}$ by covering tranformations and on $V$
by $\rho$.
The holonomy of $E_\rho$ is $\rho$, so the two constructions
give an equivalence of flat bundles and representations of $\pi$.
 
If $K$ is not connected then it is simpler to work with
flat bundles. One then defines the holonomy as a
representation of the direct sum of $\pi_1$ of the
components of $K$. In this way, the equivalence of
flat bundles and representations is recovered.
 
Suppose now that one has on each fibre of $E$ a positive density
which is locally constant on $K$.
In terms of $\rho_E$ this assumption just means 
$\mid\det\rho_E\mid=1$.
Let $V$ denote the fibre of $E$.
Then the cochain complex $C^i(K;E)$ with coefficients in $E$
can be identified with the direct sum of copies
of $V$ associated to each $i$-cell $\sigma$ of $K$.
 The identification is achieved by choosing a basepoint in each
component of $K$ and a basepoint from each $i$-cell.
By choosing a flat density on $E$ we obtain a
preferred density $\Delta_
i$ on $C^i(K,E)$. One defines the
R-torsion of $(K,E)$ to be
$\tau(K;E)=\tau(C^*(K;E),\Delta_i)\in(0,\infty)$.

\subsection[The Reidemeister zeta function and Reidemeister torsion]{The Reidemeister type zeta functions and the Reidemeister torsion of the mapping Torus.}
 
Let $f:X\rightarrow X$ be a homeomorphism of
a compact polyhedron $X$.
Let $T_f := (X\times I)/(x,0)\sim(f(x),1)$ be the
mapping torus of $f$.

We shall consider the bundle $p:T_f\rightarrow S^1$
over the circle $S^1$.
We assume here that $E$ is a flat, complex vector bundle with 
finite dimensional fibre and base $S^1$. We form its pullback $p^*E$
over $T_f$.
 Note that the vector spaces $H^i(p^{-1}(b),c)$ with
$b\in S^1$ form a flat vector bundle over $S^1$,
which we denote $H^i F$. The integral lattice in
$H^i(p^{-1}(b),\R)$ determines a flat density by 
the condition
that the covolume of the lattice is $1$.
We suppose that the bundle $E\otimes H^i F$ is acyclic for all
$i$. Under these conditions D. Fried \cite{fri2} has shown that the bundle
$p^* E$ is acyclic, and we have
\begin{equation}
 \tau(T_f;p^* E) = \prod_i
 \tau(S^1;E\otimes H^i F)^{(-1)^i}.
\end{equation}
Let $g$ be the preferred generator of the group
$\pi_1 (S^1)$ and let $A=\rho(g)$ where
$\rho:\pi_1 (S^1)\rightarrow GL(V)$.
Then the holonomy around $g$ of the bundle $E\otimes H^i F$
is $A\otimes f^*_i$.
Since $\tau(E)=\mid\det(I-A)\mid$ it follows from (16)
that
\begin{equation}
 \tau(T_f;p^* E) = \prod_i \mid\det(I-A\otimes f^*_i)\mid^{(-1)^i}.
\end{equation}
We now consider the special case in which $E$ is one-dimensional,
so $A$ is just a complex scalar $\lambda$ of modulus one.
 Then in terms of the rational function $L_f(z)$ we have \cite{fri2}:
\begin{equation}
 \tau(T_f;p^* E) = \prod_i \mid\det(I-\lambda .f^*_i)\mid^{(-1)^i}
             = \mid L_f(\lambda)\mid^{-1}.
\end{equation}

From this formula and Theorem \ref{fingenabelian}
we have
 
\begin{teo}
Let $\phi:G\to G$ be an automorphism of a finitely generated abelian group $G$. If $G$ is infinite then one has
$$
 \tau\left(T_{\hat{\phi}};p^*E\right)
 =
 \mid L_{\hat{\phi}}(\lambda) \mid^{-1}
 =
 \mid Z_{\phi}(\sigma\lambda) \mid^{(-1)^{r+1}},
 $$
and if $G$ is finite one has
$$
 \tau\left(T_{\hat{\phi}};p^*E\right)
 =
 \mid L_{\hat{\phi}}(\lambda) \mid^{-1}
 =
 \mid Z_{\phi}(\lambda) \mid^{-1}.
$$
where $\lambda$ is the holonomy of the one-dimensional
flat complex bundle $E$ over $S^1$, $r$ and $\sigma$ are the constants described in Theorem \ref{fingenabelian} .
\end{teo}

\subsection{Examples}
\begin{ex}\label{Kazhdan}
Let $\Gamma$ be a locally compact group. The following statements
are equivalent(see \cite{BHV}): 
i) $ \Gamma$ has Kazhdan's Property (T);
(ii) $1_ {\Gamma}$ is isolated in $\wh \Gamma$;
(iii) every finite dimensional irreducible unitary representation of $\Gamma$ is isolated in $\wh \Gamma$;
(iv) some finite dimensional irreducible unitary representation of $\Gamma$ is isolated in $\wh\Gamma$.

This implies immediately that  for an endomorphism of a locally compact group $\Gamma$  with Kazhdan's Property (T) the following zeta functions  coincide: $RT^f_\phi(z)= AM^f_\phi(z)$.
\end{ex}
Now let us present some examples of  Theorem \ref{finite}
for  discrete groups with extreme properties.
Suppose, an infinite discrete group
$G$ has a finite number of conjugacy classes.
Such examples can be found in \cite{crelle}.

\begin{ex}\label{ex:osingroup}
For the Osin group (see \cite{Osin}) there is only trivial(1-dimensional)
finite-dimensional representation.
Indeed, the  Osin group is an infinite finitely generated group $G$ with exactly two conjugacy classes.
All nontrivial elements of this group $G$ are conjugate. So, the group $G$
is simple, i.e. $G$ has no nontrivial normal subgroup.
This implies that group $G$ is not residually finite
(by definition of residually finite group). Hence,
it is not linear (by Mal'cev theorem)
and has no finite-dimensional irreducible unitary
representations with trivial kernel. Hence, by simplicity of $G$, it has no
finite-dimensional
irreducible unitary representation with nontrivial kernel, except for the
trivial one.
Let us remark that the Osin group is non-amenable, contains the free
group in two generators $F_2$,
and has exponential growth.

Let  $\phi:G\rightarrow G$ be any endomorphism of Osin group $G$.
Thus, we have the following:
$ RT^f(\phi^n)=RT^{ff}(\phi^n)=1$ for all $n$. This implies that for any  endomorphism of Osin group $G$
zeta functions $$ RT^{f}_\phi(z)=RT^{ff}_\phi(z)=\frac{1}{1-z} $$ are rational.

\end{ex}

\begin{ex}\label{ex;ivanovgroup}
For large enough prime numbers $p$,
the first examples of finitely generated infinite periodic groups
with exactly $p$ conjugacy classes were constructed
by Ivanov as limits of hyperbolic groups.
The Ivanov group $G$ is an infinite periodic
2-generator  group, in contrast to the Osin group, which is torsion free.
The Ivanov group $G$ is also a simple group see \cite{crelle}.
 The discussion can be completed
in the same way as in the case of the Osin group.
\end{ex}

\begin{ex}
G.~Higman, B.~H.~Neumann, and H.~Neumann
proved that any locally infinite countable group $G$
can be embedded into a countable group $G^*$ in which all
elements except the unit element are conjugate to each other
(see  \cite{serrtrees}).
The discussion above related to the Osin group remains valid for $G^*$
groups.
\end{ex}

\section{Reduction to subgroups and quotient groups}

\subsection{Reduction to subgroups}
The following lemma is useful for calculating Reidemeister numbers and zeta functions.
It will also be used in the proofs of the theorems of this chapter.

\begin{lem}\label{lem:subRei}
Let $\phi:G\to G$ be any endomorphism
 of any group $G$, and let $H$ be a subgroup
 of $G$ with the properties
$$
  \phi(H) \subset H
$$
$$
  \forall x\in G \; \exists n\in \N \hbox{ such that } \phi^n(x)\in H.
$$
Then
$$
 R(\phi) = R(\phi_H),
$$
 where $\phi_H:H\to H$ is the restriction of $\phi$ to $H$.
 If all the numbers $R(\phi^n)$ are finite then
 $$  
  R_\phi(z) = R_{\phi_H}(z).
 $$

\end{lem}

\begin{proof}
Let $x\in G$. Then there is an $n$ such that $\phi^n(x)\in H$.
It is known that $x$ is $\phi$-conjugate
 to $\phi^n(x)$ (see Lemma 7 \cite{fh}).
This means that the $\phi$-conjugacy class $\{x\}_\phi$
 of $x$ has non-empty intersection with $H$.

Now suppose that $x,y\in H$ are $\phi$-conjugate,
 ie. there is a $g\in G$ such that
 $$gx=y\phi(g).$$
We shall show that $x$ and $y$ are $\phi_H$-conjugate,
 ie. we can find a $g\in H$ with the above property.
First let $n$ be large enough that $\phi^n(g)\in H$.
Then applying $\phi^n$ to the above equation we obtain
 $$ \phi^n(g) \phi^n(x) = \phi^n(y) \phi^{n+1}(g). $$
This shows that $\phi^n(x)$ and $\phi^n(y)$ are $\phi_H$-conjugate.
On the other hand, one knows by Lemma 7 that $x$ and $\phi^n(x)$ are
 $\phi_H$-conjugate, and $y$ and $\phi^n(y)$ are $\phi_H$ conjugate,
 so $x$ and $y$ must be $\phi_H$-conjugate.

We have shown that the intersection with $H$ of a 
 $\phi$-conjugacy class in $G$ is a $\phi_H$-conjugacy class
 in $H$.
We therefore have a map
$$
\begin{array}{cccc}
 Rest : & \cR(\phi) & \to & \cR(\phi_H)\\
        & \{x\}_\phi & \mapsto & \{x\}_\phi \cap H
\end{array}
$$
This clearly has the two-sided inverse
$$
 \{x\}_{\phi_H} \mapsto \{x\}_\phi.
$$
Therefore $Rest$ is a bijection and $R(\phi)=R(\phi_H)$.
\end{proof}
\medskip

Let $Z(\phi)$ be one of the numbers
 $RT(\phi),  RT^f(\phi),
  RT^{ff}(\phi)$. We shal write $ \mathcal Z(\phi)$ for one of the corresponding sets $\mathcal{RT}(\phi),  \mathcal{RT}^f(\phi),
  \mathcal{RT}^{ff}(\phi)$ of equivalence classes of irreducible representations.
\begin{lem}\label{lem:subrep}
Let $\phi:G\to G$ be any endomorphism
 of  Abelian-by-finite group $G$, and let $H$ be a subgroup
 of $G$ with the properties
$$
  \phi(H) \subset H
$$
$$
  \forall x\in G \; \exists n\in \N \hbox{ such that } \phi^n(x)\in H.
$$
Then
$$
 Z(\phi) = Z(\phi_H),
$$
 where $\phi_H:H\to H$ is the restriction of $\phi$ to $H$.
 If all the numbers $Z(\phi^n)$ are finite then
 $$ 
 Z_{\phi}(z)= Z_{\phi_H}(z).
 $$
\end{lem}

\begin{proof}

All irreducible representations of Abelian-by-finite group $G$ are finite dimensional.

Let $\rho : G\rightarrow U(V)$ be irreducible representation, and suppose that there is a
matrix $M \in U(V)$ with
$$\rho\circ\phi= M\cdot\rho\cdot M^{-1},$$ i.e $ \rho \in\mathcal Z(\phi).$
Suppose  that $\rho_H$, the restriction of $\rho$ to $H$, is a reducible representation i.e.
there is a decomposition $ V=V_1\oplus V_2$ into $H$-modules. We shall derive a contradiction.
We can find $g\in G$ such that $\rho(g)V_1 \not\in V_1$. However for sufficiently large $n$ we have
$\phi^n(g)\in H$. This shows that $\rho(H)M^nV_1 \not\in M^nV_1$. However since $H$ is $\phi$-invariant
 and $V_1$ is $H$-invariant, $M^nV_1=V_1$. Therefore $\rho(H)M^nV_1 \not\in V_1$, which gives us a
 contradiction. Consequently $\rho$ must be an irreducible representation of $H$ on $M^nV$. However
 $M^nV=V$, so the representation $\rho_H$ is irreducible. Clearly the class of $\rho_H$ is the same as the
 class of $\rho_H\circ\phi_H$, i.e $ \rho_H \in\mathcal Z(\phi_H).$ We thus have a map
 $$ Rest: \mathcal Z(\phi)\rightarrow  \mathcal Z(\phi_H), \\ 
 \rho  \rightarrow \rho_H .$$ 
 Now let $ \rho_H \in\mathcal Z(\phi_H)$ be given. Then there is a matrix $M$ such that
$$\rho\circ\phi_H= M\cdot\rho\cdot M^{-1}.$$
If $M'$ is any other such matrix then $ M'\cdot M^{-1}$ commutes with $ \rho_H(x)$  for all $x$.        
It follows that for $g\in \phi^{-n}(H)$ the element
$$ M^{-n}\cdot\rho(\phi^n(g))\cdot M^{n}$$  is independent of the chosen $M$, and depends only on $\rho, g$
and $n$. Now suppose that $\phi^n(g)=h_1\in H$ and $\phi^m(g)=h_2\in H, m > n$.
Then $\phi^{m-n}(h_1)=h_2$, which implies
$$M^{m-n}\cdot\rho(h_1)\cdot M^{n-m}=\rho(h_2),$$
and therefore 
$$M^{-n}\cdot\rho(\phi^n(g))\cdot M^{n}=M^{-m}\cdot\rho(\phi^m(g))\cdot M^{m}.$$
The above expression is thus independent of $M$ and $n$, and depends only on $\rho$
and $g$. We may therefore define for $g\in G$

$$\bar\rho(g)=M^{-n}\cdot\rho(\phi^n(g))\cdot M^{n}$$
where $n$ is large enough that $\phi^n(g)\in H$. One can easily checks that $\bar\rho$ is a
representation of $G$. Since $\rho_H$ is irreducible it follows that $\bar\rho$ is irreducible.
One sees immediately that  the class of $\bar\rho$ is the same as the
 class of $\bar\rho\circ\phi$, i.e $ \rho \in\mathcal Z(\phi).$
 Finally we have $$Rest(\bar\rho)=\rho$$ and since any other extension $\tilde\rho$ of $\rho$ to $G$
 such that $\tilde\rho \in \mathcal Z(\phi)$ 
 must satisfy $$\tilde\rho(g)=M^{-n}\cdot\rho(\phi^n(g))\cdot M^{n},$$ we have 
 $$\overline{Rest(\rho)}=\rho.$$
 
 This shows that Rest is a  bijection, so $\mathcal Z(\phi)=  \mathcal Z(\phi_H).$
 \end{proof}

\begin{cor}
 Let $H=\phi^n(G)$. Suppose that all the numbers $R(\phi^k)$  and $Z(\phi^k)$, $k\in N$ are finite. Then $$R(\phi) = R(\phi_H), Z(\phi) = Z(\phi_H),
 R_\phi(z) = R_{\phi_H}(z), Z_{\phi}(z)=Z_{\phi_H}(z)$$.
\end{cor}

 \begin{teo}\label{teo:main1}
Let $\phi:G\to G$ be any endomorphism
 of  Abelian-by-finite group $G$, and let $H$ be a subgroup
 of $G$ with the properties
$$
  \phi(H) \subset H
$$
$$
  \forall x\in G \; \exists n\in \N \hbox{ such that } \phi^n(x)\in H.
$$
Suppose that all the numbers $R(\phi^k)$  and $Z(\phi^k)$, $k\in N$ are finite.
If one  of the following conditions is satisfied:
 \begin{itemize}
\item [1.] $H$ is a finitely generated abelian group, or
\item [2.] $H$ is a finite group, or
\item [3.] $H$ is a crystallographic group 
with diagonal holonomy $\Z_2$ and $\phi_H$ is an automorphism,
\end{itemize}
then 
 $$ 
 R_\phi(z) = R_{\phi_H}(z)=Z_{\phi_H}(z)=Z_{\phi}(z)
 $$
 and  these zeta functions are rational functions.
 
\end{teo}
\begin{proof}
 TBFT (resp., TBF$T_f$, TBF$T_{ff}$) for
an endomorphism $\phi:G\to G$ and its iterations were proven in \cite{fh} for finitely generated
abelian groups and for  finite groups.
Any crystallographic group 
with diagonal holonomy $\Z_2$ is 
a polycyclic-by-finite group and it has only finite dimensional irreducible representations. In \cite{RJMP, crelle} twisted Burnside-Frobenius theorem($TBFT_f$ and $TBFT_{ff}$) was proven for automorphisms of polycyclic-by-finite groups.

 These results implie equality of the Reidemeister zeta function  $R_{\phi_H}(z)$ and the zeta function $Z_{\phi_H}(z)$.
Hence from Lemma \ref{lem:subRei} and Lemma \ref{lem:subrep} it follows
that
$$ 
 R_\phi(z) = R_{\phi_H}(z)=Z_{\phi_H}(z)=Z_{\phi}(z).
 $$

In \cite{fh} the rationality of the Reidemeister zeta function $R_\phi(z)$
was proven for endomorphisms of finitely generated abelian groups and for finite groups and in
\cite{DekTerBus} the rationality of $R_\phi(z)$ was proven for  automorphisms of almost-crystallographic groups with diagonal holonomy $Z_2$. This completes the proof.

\end{proof}
\medskip
\subsection{ Reduction to Injective Endomorphisms on quotient groups}
Let $G$ be a group and $\phi : G\rightarrow G$ an endomorphism.
We shall call an element $x\in G $ nilpotent if there is
an $n\in \N $ such that $\phi^n(x)= e$. 

Let $N$ be the set 
of all nilpotent elements of $G$.

 Let $Z(\phi)$ be one of the numbers
 $  RT^f(\phi),
  RT^{ff}(\phi)$ and  $ \mathcal Z(\phi)$ one of  the corresponding sets of equivalence classes of irreducible representations.
 
 \begin{teo}
The set $N$ is a normal subgroup of $G$. We have $ \phi(N) \subset N$ and 
$\phi^{-1}(N)=N$. 
Thus $\phi$ induces an endomorphism $[\phi/N]$ of the quotient group $G/N$ given by $[\phi/N](xN):=\phi(x)N$.
The endomorphism  $[\phi/N]: G/N \rightarrow G/N$ is injective,
and we have 
$$
R(\phi) = R([\phi/N]),\,\, \,\,Z(\phi) = Z([\phi/N]).
$$
Let the numbers $R(\phi^n)$ and  $Z(\phi^n)$ be all finite. 
Then
$$
R_\phi(z) = R_{[\phi/N]}(z), \,\, \,\, Z_{\phi}(z)=Z_{[\phi/N]}(z).
$$
If the quotient group $G/N$ is polycyclic then one has the following Gauss congruences
for Reidemeister numbers:
 $$
 \sum_{d\mid n} \mu(d)\cdot R(\phi^{n/d}) \equiv 0 \mod n
 $$
for any $n$.
If  one of  the following conditions is satisfied:
 \begin{itemize}
\item [1.] The quotient group $G/N$ is a finitely generated abelian group, or
\item [2.] $G/N$ is a finite group, or
\item [3.] $G/N$ is  a finitely generated torsion free nilpotent group, or
\item [4.] $G/N$ is a crystallographic group 
with diagonal holonomy $\Z_2$ and $[\phi/N]$ is an automorphism,
\end{itemize}
then 
 $$ 
 R_\phi(z) = R_{[\phi/N]}(z)=Z_{[\phi/N]}(z)=Z_{\phi}(z)
 $$
 and  these zeta functions are rational functions.
\end{teo}

\begin{proof}

(i) Let $x\in N, g\in G$. Then for some $n\in \N$ we have $\phi^n(x)= e $.
Therefore $\phi^n(gxg^{-1})=\phi^n(gg^{-1})= e$. This shows that $gxg^{-1}\in N$
so $N$ is a normal subgroup of $G$.

(ii) Let  $x\in N$ and choose $n$ such that $\phi^n(x)=e$. Then  $\phi^{n-1}(\phi(x))=e$ so $\phi(x)\in N$. Therefore  $ \phi(N) \subset N$.

(iii) If  $\phi(x)\in N$ then there is an $n$ such that $\phi^n(\phi(x))=e$. Therefore $\phi^{n+1}(x)=e$ so  $x\in N$. This show that
 $\phi^{-1}(N)\subset N$. The converse inclusion follows from (ii).

(iv) We shal write $ \cR(\phi)$ for the set of $\phi$-conjugacy classes
of elements of $G$. We shall now  show that the map $x\rightarrow xN$
induces a bijection  $ \cR(\phi)\rightarrow \cR([\phi/N])$.
Suppose $x,y\in G$ are $\phi$-conjugate. Then there is a $g\in G$ with
$gx=y\phi(g).$ Projecting to the quotient group $G/N$ we have
$gnxN=yN\phi(g)N$, so $gNxN=yN[\phi/N](gN)$. This means that $xN$ and $yN$
are $[\phi/N]$-conjugate in $G/N$. Conversely suppose that $xN$ and $yN$
are  $[\phi/N]$-conjugate in $G/N$. Then there is a $gN \in G/N$
such that  $gNxN=yN[\phi/N](gN)$. In other words $gx\phi(g)^{-1}y^{-1})=e.$
Therefore $\phi^n(g)\phi^n(x)=\phi^n(y)\phi^n(\phi(g)).$

This shows that $\phi^n(x)$ and  $\phi^n(y)$ are $\phi$-conjugate.
However  $x$ and  $\phi^n(x)$ are $\phi$-conjugate as are
 $y$ and  $\phi^n(y)$. Therefore $x$ and $y$ are  $\phi$-conjugate.

(v) We have shown that $x$ and $y$ are $\phi$-conjugate iff $xN$ and $yN$
are  $[\phi/N]$-conjugate. From this it follows that $x\rightarrow xN$
induces a  bijection from  $ \cR(\phi)$ to $ \cR([\phi/N])$.
Therefore $R(\phi)=R([\phi/N])$.

(vi) We shall now show that $Z(\phi)=Z([\phi/N])$. Let $\rho \in \mathcal Z(\phi)$
and let $M$ be a transformation for which
$$ \rho\circ\phi= M\cdot\rho\cdot M^{-1}$$

if $x\in N$ then there is an $n\in \mathbb N$ with $\phi^n(x)=e$. Thus $ N$ is contained in the kernel of $\rho$ and there is a representation $[\rho/N]$ of $G/N$ given by
$$[\rho/N](gN):= \rho(g).$$
Since $[\rho/N]$  satisfies identity 
 $$ [\rho/N]\circ[\phi/N]= M\cdot[\rho/N]\cdot M^{-1},$$
we have $[\rho/N]\in \mathcal Z([\phi/N])$.

(vii) Conversely if $\rho \in \mathcal Z([\phi/N])$ we may construct a $\bar\rho \in \mathcal Z(\phi)$ by
$$ \bar\rho(x):= \rho(xN).$$
It is clear that 
$ \overline{[\rho/N]}=\rho$ and $ \bar\rho/N = \rho$.

In \cite{RJMP, crelle} the  twisted Burnside-Frobenius theorem($TBFT_f$ and $TBFT_{ff}$) was proven for endomorphisms of polycyclic groups  and for automorphisms of polycyclic-by-finite groups. 
Then from (v) and (vi) it follows that 
$$R(\phi^n)=R([\phi/N]^n)=Z([\phi/N]^n)= Z(\phi^n).$$
 Gauss congruences now follow from Corollary \ref{cor:periodicanddyn} and the
general theory of congruences for periodic points
(cf. \cite{Smale,Z}).
More precisely, let $P_d$ be the number of 
periodic points of least period $d$
of the dynamical system of Corollary \ref{cor:periodicanddyn}.
Then $R(\phi^n)= Z(\phi^n)=\sum\limits_{d|n} P_d$. 

By the 
M\"obius inversion formula,
$$
\sum\limits_{d|n} \mu(d) R(\f^{n/d})=P_n \equiv 0
\mod n,
$$
since number $P_n$ is always divisible by $n$,
because $P_n$ is exactly $n$ times the number of
 orbits of cardinality $n$.

Twisted Burnside-Frobenius theorem($TBFT_f$ and $TBFT_{ff}$) implies also equality of Reidemeister zeta function  $R_\phi(z)$ and zeta function $RT^f_\phi(z)=RT^{ff}_\phi(z))$. In \cite{Fel00} the rationality of the Reidemeister zeta function $R_\phi(z)$
was proven for endomorphisms of finitely generated abelian groups and for endomorphisms of  finitely generated  torsion free nilpotent  groups, and in
\cite{DekTerBus} the rationality of $R_\phi(z)$ was proven for  automorphisms of crystallographic groups with diagonal holonomy $Z_2$. This completes the proof.

\end{proof}

\section{P\'olya -- Carlson dichotomy for Reidemeister zeta function}

In this section we present results  in support of
a P{\'o}lya--Carlson dichotomy between rationality and a
natural boundary for the analytic behaviour of  the Reidemeister zeta function
of Abelian group endomorphisms.

Let $\phi: G \rightarrow G$  be a endomorphism of a  countable  Abelian group  $G$
that is a subgroup  of~$\mathbb{Q}^d$, where $d\geqslant 1$.
Let  $R=\mathbb{Z}[t]$ be a polynomial  ring. Then the Abelian group $G$ naturally carries the structure  of a $R$-module over the ring $R=\mathbb{Z}[t]$  where multiplication by~$t$
corresponds to application of the endomorphism:
$ tg=\phi(g)$  and extending this in a natural way to polynomials.
That is, for $g\in G$ and $ f=\sum_{n\in \mathbb{Z} }c_n t^n \in R=\mathbb{Z}[t]$ set
$$ fg= \sum_{n\in \mathbb{Z} }c_n t^ng=\sum_{n\in \mathbb{Z} }c_n \phi^n(g),$$
where all but finitely many $c_n\in \mathbb{Z}$ are zero.
This is a standard procedure for the
study of dual automorphisms of compact Abelian
groups, see Schmidt~\cite{Sch} for an overview.

 Let us now consider the Pontryagin dual  group $\wh G$ and dual endomorphism
 $\widehat{\phi}:\rho\mapsto \rho\circ\phi$ on the  $\wh G$.
 
We shall require the following statement:
 
\begin{lem}\cite{fh}\label{dual}
Let $\phi:G\to G$ be an endomorphism of an Abelian group $G$.
Then the kernel $\ker\left[\hat{\phi}:\hat{G}\to\hat{G}\right]$
is canonically isomorphic to the
Pontryagin dual of $\coker\phi$.
\end{lem}
 
\begin{proof}

We construct the isomorphism explicitly.
Let $\chi$ be in the dual of $\coker(\phi:G\to G)$.
In that case $\chi$ is a homomorphism
  $$  \chi : G / \im(\phi) \longrightarrow U(1).$$
There is therefore an induced map
 $$  \overline{\chi} : G \longrightarrow U(1) $$
which is trivial on $\im(\phi)$.
This means that $\overline{\chi}\circ\phi$ is trivial, or
in other words $\hat{\phi}(\overline{\chi})$ is the identity element of $\hat{G}$.
We therefore have $\overline{\chi}\in\ker(\hat{\phi})$.
 
If on the other hand we begin with $\overline{\chi}\in\ker(\hat{\phi})$,
then it follows that $\chi$ is trivial on $\im\phi$, and so
$\overline{\chi}$ induces a homomorphism
 $$ \chi : G / \im(\phi) \longrightarrow U(1)$$
and $\chi$ is then in the dual of $\coker\phi$.
The correspondence $\chi\leftrightarrow\overline{\chi}$ is
clearly a bijection.
 
\end{proof}

 The following results are also needed
 
\begin{lem}\cite{Mi}\label{miles1}
 Let $L \subset N$ be $R$-modules and $g\in R$.\\
 Then
 \\
(1) $$ \Bigl |\frac{N}{gN}\Bigr | = \Bigl |\frac{N/L}{g(N/L)}  \Bigr |\Bigl |\frac{L}{L \cap gN} \Bigr |$$\\
 (2) If $N/L$ is finite and the map $x\to gx$ is a monomorphism of $N$ then
 $$\Bigl |\frac{N}{gN}\Bigr |=\Bigl |\frac{L}{gL}\Bigr |.$$
 \end{lem}

Suppose that $G$ as an $R$- module  satisfies the following conditions:

(1) the set of associated primes $Ass(G)$ is finite and consists entirely of non-zero
principal ideals of $R$,

(2) the map $g \rightarrow (t^j -1)g$  is a monomorphism of $G$ for all $j\in \mathbb{N}$\\
(equivalently, $t^j - 1 \notin \mathfrak{p}$ for all $ \mathfrak{p}\in Ass(G)$ and all $j\in
\mathbb{N}$),

(3) for each $\mathfrak{p}\in Ass(G)$, $ m(\mathfrak{p})= \dim_{\mathbb{K}(\mathfrak{p})}G_{\mathfrak{p}} < \infty$,
where $\mathbb{K}(\mathfrak{p})$ denotes the field of fractions of $R/\mathfrak{p}$ and $G_{\mathfrak{p}} = G\otimes_R\mathbb{K}(\mathfrak{p})$ is the localization of the module $G$ at $\mathfrak{p}$.

\begin{lem}\cite{Mi}\label{miles2}
 Let N be an $R$-module for which $Ass(N)$ consists of finitely many non-trivial
 principal ideals and suppose $ m(\mathfrak{p})= \dim_{\mathbb{K}(\mathfrak{p})}N_{\mathfrak{p}} < \infty$. If $g\in R$ is such that the map $x \to gx$ is a monomorphism of $N$, then
 $N/gN$ is finite.
 
  \end{lem}
 
If the Pontryagin dual endomorphism $\widehat{\phi}$ is an ergodic finite entropy
epimorphism of the  compact connected Abelian group $\wh G$ 
of  finite  dimension $d\geq 1$ then the endomorphism $\phi: G \rightarrow G$ satisfies the conditions (1) - (3) above. Such  dual groups $\wh G$  are called  solenoids(see \cite{Mi, Sch}).

For the dual endomorphism~$\widehat{\phi}:\wh G\rightarrow \wh G$,
we use the following closed periodic point counting formula  taken
from~\cite[Th.~1.1]{Mi} and \cite[Pr.~14]{BMW}. Let ~$F_{\widehat{\phi}}(j)=|Fix(\widehat\phi^j)|$
denotes the number of points fixed by the
endomorphism~$\widehat{\phi}^j$. Some obvious conditions such as ergodicity and finite entropy are 
necessary to ensure that $F_{\widehat{\phi}}(j)$ is finite for all $j\in \mathbb{N}.$  Let $\mathcal{P}(\mathbb{K})$ denote the
set of finite places of the algebraic number field~$\mathbb{K}$. The places
 of the field $\mathbb{K}$ are the equivalence classes of absolute
values on $\mathbb{K}$. When $char(\mathbb{K})=0$, the infinite places are the archimedean ones.
All other places are said to be finite. Given a finite place of $\mathbb{K}$, there
corresponds a unique discrete valuation $v$ whose precise value group is $\mathbb{Z}$.
The corresponding normalised absolute value $|\cdot|_v= |\mathcal{R}_v|^{-v(\cdot)}$,
where $\mathcal{R}_v$ is the residue class field of $v$.
For any set
of places~$S$, we
write~$|x|_S=\prod_{v\in S}|x|_v$.

\begin{prop}\cite[Th.~1.1]{Mi}, \cite[Pr.~14]{BMW}\label{main_formula1}
If~$\widehat{\phi}:\wh G\rightarrow \wh G$ is an ergodic finite entropy automorphism of a finite
dimensional compact connected Abelian group $\wh G$ , then there exist
algebraic number fields~$\mathbb{K}_1,\dots,\mathbb{K}_n$, sets
of finite places~$P_i\subset \mathcal{P}(\mathbb{K}_i)$ and
elements~$\xi_i\in \mathbb{K}_i$, no one of which is a root of
unity for~$i=1,\dots,n$, such that for any $j\in \mathbb{N}$.
 
\begin{equation}\label{periodic}
F_{\widehat{\phi}}(j)=\prod_{i=1}^{n}\prod_{v\in P_i}
|\xi_i^j-1|_{v}^{-1} = \prod_{i=1}^{n}|\xi_i^j-1|_{P_i}^{-1}.
\end{equation}
\end{prop}

\begin{proof}

 We outline  the major steps in the proof.
 
 Under assumptions of the proposition the number of the periodic points $F_{\widehat{\phi}}(j)$ is finite for all $j\in \mathbb{N}$. Considering abelian group $G$ as $Z[t]$-module and using a straightforward duality argument in  Lemma \ref{dual}( or
in  \cite[Lemma 7.2]{LSW}) we have
\begin{eqnarray*}
 F_{\widehat{\phi}}(j)=|Fix(\widehat\phi^j) |= | \Ker(\widehat\phi^j-\Id_{\widehat G}) | = 
|\widehat{\coker(\phi-\Id_G)}| =\\
= |\coker(\phi^j - \Id_G) | = |G/(\phi^j-1)G| = |G/(t^j-1)G|. 
\end{eqnarray*}
 
The multiplicative set $ U = \bigcap_{\mathfrak{p}\in Ass(G)} R-{\mathfrak{p}}$
has $U\cap ann(a) = \emptyset$ for all non-zero $a\in G$, so the natural map $G\to U^{-1}G$
is a monomorphism. Identifying localizations of $R$ with subrings of $\mathbb{Q}(t)$, the
domain $ \mathfrak{R}=U^{-1}R = \bigcap_{\mathfrak{p}\in Ass(G)} R_\mathfrak{p}$ is a finite intersection
of discrete valuation rings and is therefore a principal ideal domain \cite{Matsumura}.
The assumptions of finite entropy and finite topological dimension force $U^{-1}G $ to be a Noeterian
$ \mathfrak{R}$ - module.Hence, there is a prime filtration
$$\{0\} = G_0 \subset G_1\subset \cdot \cdot\cdot \subset G_n = U^{-1} G $$
in which $ G_i/G_{i-1} \cong \mathfrak{R}/\mathfrak{q_i}$
for non-trivial primes $\mathfrak{q_i} \subset \mathfrak{R}, 1\leq i\leq n$.
Moreover, $\mathfrak{p_i}=\mathfrak{q_i}\cap R\in Ass(G)$ for all $1\leq i\leq n$.
Identifying $G$ with its image in $U^{-1}G$ and intersecting the chain above with $G$
gives a chain 

$$\{0\} = L_0 \subset L_1\subset \cdot \cdot \cdot\subset L_n = G .$$
Considering this chain of $R$-modules, for each $1\leq i\leq n$ there is an induced inclusion
$$
 \frac{L_i}{L_{i-1}} \hookrightarrow    \frac{G_i}{G_{i-1}} \cong \frac{\mathfrak{R}}{\mathfrak{q_i}}\cong
 \mathbb{K(\mathfrak{p_i})}= K_i
$$
and $N_i = L_i/L_{i-1}$ may be considered as a fractional ideal of $E_i = R/\mathfrak{p_i}$.
Using Lemma \ref{miles1}(1),
$$ \Bigl|\frac{L_i}{(t^j - 1)L_i}\Bigr | = \Bigl|\frac{N_i}{(t^j - 1)N_i} \Bigr |\Bigl|\frac{L_{i-1}}{L_{i-1} \cap (t^j-1)L_i} \Bigr |,$$
where $1\leq i\leq n$. Let $y\in L_i$, let $\eta$ denote the image of $y$ in $N_i$ and let $\xi_i$
denote the image of $t$ in $E_i$. If $(t^j-1)y\in L_{i-1}$ then $(\xi_i^j-1)\eta = 0$.
The ergodicity assumption implies $ t^j-1\notin \mathfrak{p_i}$ so $(\xi_i^j-1)\neq 0$. 
Therefore, $\eta=0$ and $y\in L_{i-1}$. It follows that $L_{i-1} \cap (t^j-1)L_i = (t^j-1)L_{i-1}$
and hence,

$$ \Bigl| \frac{L_i}{(t^j-1)L_i}\Bigr | = \Bigl|\frac{N_i}{(t^j-1)N_i}\Bigr |\Bigl|\frac{L_{i-1}}{ (t^j-1)L_{i-1}} \Bigr |,$$

Successively applying this formula to each of the modules $L_i,1\leq i\leq n$, gives,
$$ | G/(t^j-1)G| = \prod_{i=1}^{n}|N_i/(t^j-1)N_i|$$

Consider now an individual  term $|N_i/(t^j-1)N_i|$. Since $E_i$ is a finitely generated domain,
\cite[Th.~1.1]{Mi}\cite[Th. 4.14]{Eisenbud} shows that the integral closure $D_i$
of $E_i$ in $K_i$ is a finitely generated Dedekind domain. Therefore, $D_i$ is finitely generated as an $E_i$-module. We may consider $I_i = D_i\otimes_{E_i}N_i$ as a
fractional ideal of $D_i$. Lemma \ref{miles2} and  Lemma \ref{miles1}(2) imply that
$|N_i/(\xi_i^j-1)N_i|= |I_i/(\xi_i^j-1)I_i|$ (see \cite{Mi}).
By considering $I_i/(\xi_i^j-1)I_i$ as a $D_i$-module, finding a composition series for this module and    successively localizing at each of its associated primes to obtain multiplicities, it follows that
$$|I_i/(\xi_i^j-1)I_i| = \prod_{\mathfrak{m}\in Ass(I_i/(\xi_i^j-1)I_i)}q_{\mathfrak{m}}^{\delta_\mathfrak{m}(\xi_i,I_i)},$$ where $q_{\mathfrak{m}}= |D_i/\mathfrak{m}|$ and $\delta_\mathfrak{m}(\xi_i,I_i) = \dim_{D_i/\mathfrak{m}}(I_i/(\xi_i^j-1)I_i)_\mathfrak{m}$. Let 
$$P_i=\{\mathfrak{m}\in Spec(D_i) : I_{\mathfrak{m}} \neq K_i\}.$$
 It follows that the product above may be taken over all $\mathfrak{m}\in P_i$ to yield the same result. Each localization $(D_i)_{\mathfrak{m}}$ is a distinct valuation ring of $K_i$ and $P_i$ may be identified with a set of finite places of the global field $K_i$.
Hence, since $\delta_\mathfrak{m}(\xi_i,D_i) = v_\mathfrak{m}(\xi_i^j-1)$, finally we have  
$$ |I_i/(\xi_i^j-1)I_i| = \prod_{\mathfrak{m}\in P_i}q_{\mathfrak{m}}^{\delta_\mathfrak{m}(\xi_i,D_i)}= \prod_{\mathfrak{m}\in P_i}q_{\mathfrak{m}}^{v_\mathfrak{m}(\xi_i^j-1)} = \prod_{\mathfrak{m}\in P_i}|\xi_i^j-1 |_{\mathfrak{m}}^{-1},$$
where $|\cdot|_{\mathfrak{m}}$ is the normalised absolute value arising from $D_{\mathfrak{m}}$. This concludes the proof. 
 \end{proof}

\begin{rk}
It is useful to note 
that~$\mathbb{K}_i=\mathbb{Q}(\xi_i)$,~$i=1,\dots,n$. Applying the Artin product formula \cite{Weil}
to~(\ref{periodic}) gives
\begin{equation}\label{main_formula2}
F_{\widehat{\phi}}(j)=
\prod_{i=1}^{n}
|\xi_i^j-1|_{P_i^\infty\cup S_i},
\end{equation}
where~$P_i^\infty$ denotes the set of infinite places
of~$\mathbb{K}_i$ and~$S_i=\mathcal{P}(\mathbb{K}_i)\setminus
P_i$. It is also worth noting that~\cite[Rmk.~1]{Mi}
implies that~$|\xi_i|_v=1$ for all~$v\in P_i$,~$i=1,\dots,n$,
as~$\phi$ is an automorphism. 
\end{rk}

The following results are  needed to have
more ready access to the theory of linear recurrence sequences.
Relevant background on the connection between linear recurrence sequences
and the rationality may be found in the monograph of Everest, van der Poorten, Shparlinski and Ward \cite{EPSW}. 

\begin{lem}(cf. \cite{BMW})\label{generating}
Let~$R(z)=\sum_{n= 1}^{\infty}R(\phi^n)z^n$.
If~$R_\phi(z)$ is rational then~$R(z)$ is rational.
If~$R_\phi(z)$ has analytic continuation beyond
its circle of convergence, then so too does~$R(z)$.
In particular, the existence of a natural boundary
at the circle of convergence for~$R(z)$ implies the
existence of a natural boundary for~$R_\phi(z)$.
\end{lem}

\begin{proof}
This follows from the fact that~$R(z)= z\cdot R'_\phi(z)/R_\phi(z)$.
\end{proof}

One of the important links between the arithmetic
properties of the coefficients of a complex power
series and its analytic behaviour is given
by the P{\'o}lya--Carlson theorem~\cite{Car},~\cite{Po}, ~\cite{Segal}.

\medskip{\noindent P\'{o}lya--Carlson Theorem.}
{\it A power series with integer coefficients
and radius of convergence~$1$ is either rational or has the
unit circle as a natural boundary.}\medskip

For the proof of the main theorem of this section we
use the following key  result of Bell,  Miles and  Ward .

\begin{lem}[Lemma 17 in \cite{BMW}]\label{derivative}
Let~$S$ be a finite list of places of algebraic number fields
and, for each~$v\in S$, let~$\xi_v$ be a non-unit root in the
appropriate number field such that~$|\xi_v|_v=1$. Then the
function
\[
F(z)=\sum_{n= 1}^{\infty}f(n)z^n,
\]
where~$f(n)=\prod_{v\in S}|\xi_v^n-1|_v$ for~$n\ge1$, has the unit circle as
a natural boundary.
\end{lem}

The main results of this section
are the following counting formulas for the Reidemeister numbers and a P\'olya--Carlson dichotomy between rationality and a natural boundary for the analytic behaviour of the Reidemeister zeta function.
We  follow the method of the proof of  Bell,  Miles and  Ward  in ~\cite[Theorem 15]{BMW} for the Artin--Masur zeta function  of compact abelian groups automorphisms.
\begin{teo}\label{Main}
Let $\phi: G \rightarrow G$  be an automorphism of a  countable  Abelian group  $G$
that is a subgroup  of~$\mathbb{Q}^d$, where $d\geqslant 1$.
Suppose that the group $G$ as  $R = \mathbb{Z}[t]$- module  satisfies the following conditions:

(1) the set of associated primes $Ass(G)$ is finite and consists entirely of non-zero
principal ideals of   the polynomial  ring $R = \mathbb{Z}[t]$,

(2) the map $g \rightarrow (t^j -1)g$  is a monomorphism of $G$ for all $j\in \mathbb{N}$\\
(equivalently, $t^j - 1 \notin \mathfrak{p}$ for all $ \mathfrak{p}\in Ass(G)$ and all $j\in
\mathbb{N}$),

(3) for each $\mathfrak{p}\in Ass(G)$, $ m(\mathfrak{p})= \dim_{\mathbb{K}(\mathfrak{p})}G_{\mathfrak{p}} < \infty$.

Then there exist
algebraic number fields~$\mathbb{K}_1,\dots,\mathbb{K}_n$, sets
of finite places~$P_i\subset \mathcal{P}(\mathbb{K}_i)$,
 $S_i=\mathcal{P}(\mathbb{K}_i)\setminus
P_i$,
 and
elements~$\xi_i\in \mathbb{K}_i$, no one of which is a root of
unity for~$i=1,\dots,n$, such that 
\begin{equation}\label{Reidemeister1}
R(\phi^j) = \prod_{i=1}^{n}\prod_{v\in P_i}
|\xi_i^j-1|_{v}^{-1} = \prod_{i=1}^{n}|\xi_i^j-1|_{P_i}^{-1}  =
\prod_{i=1}^{n}
|\xi_i^j-1|_{P_i^\infty\cup S_i}
\end{equation}
 for all $j\in \mathbb{N}$.
 
 Suppose that the last
product in~\eqref{Reidemeister1} only involves
finitely many places and that~$|\xi_i|_v\neq 1$ for all~$v$
in the set of infinite places $P_i^\infty$ of $\mathbb{K}_i$ and all $i=1,\dots,n$. \\
Then the Reidemeister zeta function $R_{\phi}(z)$ is
either rational function or has a natural boundary at its circle of
convergence, and the latter occurs if and only if~$|\xi_i|_v=1$
for some~$v\in S_i$,~$1\le i\le n$.
\end{teo}

\begin{proof}

\medskip
 The  Reidemeister number of an endomorphism $\phi$ of 
 an Abelian group $G$  coincides with the cardinality of the  quotient group $ \coker(\phi-\Id_G)=G/{\rm Im}(\phi-\Id_G)$
(or $\coker(\Id_G -\phi)=G/{\rm Im}(\Id_G-\phi)$).  
A straightforward duality argument using Lemma \ref{dual} shows that
\begin{equation}\label{Pont}
 R(\phi)= |\coker(\phi-\Id_G) | = |\widehat{\coker(\phi-\Id_G)}| = |\Ker(\widehat\phi-\Id_{\widehat G}) | = |Fix(\widehat\phi) |. 
\end{equation}

If the endomorphism $\phi: G \rightarrow G$ satisfies the conditions (1) - (3),
then the Pontryagin dual endomorphism $\widehat{\phi}$ is an ergodic finite entropy
epimorphism of the  compact connected Abelian group $\wh G$ 
of the finite  dimension $d\geq 1$ i.e.  the Pontryagin dual group $\wh G$  is a solenoid(see \cite{Mi, Sch}). Hence the Reidemeister numbers $R(\phi^j)$ and the number of periodic points of the dual map $F_{\widehat{\phi}}(j)$  are finite for all $j\in \mathbb{N}$ .  By (\ref{periodic}), (\ref{main_formula2}) and (\ref{Pont}) we have
\begin{equation}\label{RP}
R(\phi^j)= F_{\widehat{\phi}}(j)= \prod_{i=1}^{n}\prod_{v\in P_i}
|\xi_i^j-1|_{v}^{-1} = \prod_{i=1}^{n}|\xi_i^j-1|_{P_i}^{-1}  =
\prod_{i=1}^{n}
|\xi_i^j-1|_{P_i^\infty\cup S_i}.
\end{equation}

Let~$S_i^*=\{v\in S_i:|\xi_i|_v\neq 1\}$, $S_i^{**}=\{v\in S_i:|\xi_i|_v> 1\}$ and let
$$
f(j)=\prod_{i=1}^n|\xi_i^j-1|_{S_i\setminus S_i^*},  \,\,\,\,\,\,
g(j)=\prod_{i=1}^n|\xi_i^j-1|_{P_i^\infty\cup S_i^*}.
$$
So,~$R(\phi^j)=f(j)g(j)$ by~(\ref{Reidemeister1}).
By the ultrametric property
$$
g(j)=\prod_{i=1}^n
|\xi_i|^j_{S_i^{**}}
\cdot|\xi_i^j-1|_{P_i^\infty}.
$$
  We can
expand the product over infinite places using an appropriate
symmetric polynomial  to obtain an expression of the form
\begin{equation}\label{dominant}
g(j)=\sum_{I\in\mathcal{I}}d_I w_I^j,
\end{equation}
where~$\mathcal{I}$ is a finite indexing set,~$d_I\in\{-1,1\}$
and~$w_I\in\mathbb{C}$.

Furthermore, by~\eqref{dominant},
\[
R_\phi(z)
=
\exp\left(\sum_{I\in\mathcal{I}}d_I
\sum_{j= 1}^{\infty}
\frac{f(j)(w_Iz)^j}{j}\right).
\]
If~$S_i\setminus S^*_i=\varnothing$ for all~$i=1,\dots,n$,
then~$f(j)\equiv 1$, and it follows immediately
that the  Reidemeister zeta function $R_\phi(z)$ is rational function.

Now suppose that~$S_i\setminus S^*_i\neq\varnothing$ for
some~$i$. As noted in
Lemma \ref{generating}, we need only exhibit
a natural boundary at the circle of convergence for
\[
\sum_{I\in\mathcal{I}}d_I \sum_{j= 1}^{\infty} f(j)(w_Iz)^j
\]
to exhibit one for~$R_\phi(z)$.
Moreover,~$\limsup_{j\rightarrow\infty}f(j)^{1/j}=1$, so for
each~$I\in\mathcal{I}$, the series
\[
\sum_{j= 1}^{\infty} f(j)(w_Iz)^j
\]
has radius of convergence~$|w_I|^{-1}$.

Since~$|\xi_i|_v\neq 1$ for all~$v\in
P_i^\infty$,~$i=1,\dots,n$, there is a dominant term~$w_J$ in
the expansion~\eqref{dominant}, for which
\[
|w_J|
=
\prod_{i=1}^{n}
|\xi_i|_{S^{**}_i}
\prod_{v\in P_i^\infty}
\max\{|\xi_i|_v,1\}
=
\prod_{i=1}^{n}
\prod_{v\in P_i^\infty\cup \mathcal{P}(\mathbb{K}_j)}
\max\{|\xi_i|_v,1\},
\]
and~$|w_J|>|w_I|$ for all~$I\neq J$ (note that~$\log|w_J|$ is
the topological entropy, as given by~\cite{LW}).

Since~$|w_J|^{-1}<|w_I|^{-1}$ for all~$I\neq J$, this means
that it suffices to show that the circle of
convergence~$|z|=|w_J|^{-1}$ is a natural boundary for~$\sum_{j= 1}^{\infty} f(j)(w_Iz)^j$. But this is the case precisely when~$\sum_{j= 1}^{\infty} f(j)z^j$ has the unit circle as a natural boundary, and this
has already been dealt with by
Lemma~\ref{derivative}.
\end{proof}

\subsection{Examples}
To give an example of irrational Reidemeister zeta function 
let us consider an endomorphism $\phi: g \rightarrow 2g$ on the module 
$\Z[\frac{1}{3}]$ which is an infinitely generated abelian group.
We follow  the method and the  calculations  of Everest, Stangoe and Ward in Lemma 4.1 in \cite{ESW} 
for the Artin--Masur zeta function of the dual compact
abelian group endomorphism $\widehat{\phi}$ .

\begin{lem}(cf. Lemma 4.1 of \cite{ESW} ) 
The Reidemeister zeta function $R_{\phi}(z)$
 has natural boundary $\vert z\vert=\frac{1}{2}$.
\end{lem}

\begin{proof}
The dual compact abelian group $\widehat{\Z[\frac{1}{3}]}$ is a one dimensional solenoid.
By (\ref{periodic}), (\ref{main_formula2}) and (\ref{RP}) the Reidemeister numbers of iterations of $\phi$ and the number of periodic points of the dual map $\widehat{\phi}$  are 
$R(\phi^j) = |Fix({\widehat{\phi}} ^j )|=F_{\widehat{\phi}}(j)= |2^j-1|\cdot |2^j-1|_3  .$

Let  $\xi(z)=\sum_{n=1}^{\infty}\frac{z^n}{n}
\vert2^n-1\vert\cdot
\vert2^n-1\vert_3 $  so the Reidemeister zeta function

$R_{\phi}(z)=\exp(\xi(z))$.
Now
\begin{eqnarray*}
\xi(z)&=&\sum_{n=0}^{\infty}
\frac{z^{2n+1}}{2n+1}
(2^{2n+1}-1)+\sum_{n=1}^{\infty}
\frac{z^{2n}}{2n}(2^{2n}-1)\vert2^{2n}-1\vert_3\\
&=&\log\left(\frac{1-z}{1-2z}\right)
-{\frac{1}{2}}\log\left(
\frac{1-z^2}{1-4z^2}\right)+
\sum_{n=1}^{\infty}
\frac{z^{2n}}{2n}(2^{2n}-1)\vert2^{2n}-1\vert_3.
\end{eqnarray*}
 Notice that
\begin{eqnarray*}
\vert2^n-1\vert_3=\vert
(3-1)^n-1\vert_3=\vert3^n-n3^{n-1}+\dots+(-1)^{n-1}3n+(-1)^n-1\vert_3=\\
= \begin{cases}
\frac{1}{3}\vert n\vert_3&\mbox{if }n\mbox{ is even,}\\
1&\mbox{if }n\mbox{ is odd.}
\end{cases}
\end{eqnarray*}
In
particular,
\begin{equation}\label{exp}
\vert4^n-1\vert_3=\vert2^{2n}-1\vert_3=\textstyle\frac{1}{3}\vert2n\vert_3=\frac{1}{3}\vert
n\vert_3.
\end{equation}

Write $\frac{1}{6}\xi_1(z)$ for the last term in an
expression for $\xi(z)$ above, so by~\eqref{exp}
\begin{eqnarray*}
\xi_1(z)&=&3
\sum_{n=1}^{\infty}
\frac{z^{2n}}{n}(4^{n}-1)\vert4^{n}-1\vert_3 =
\sum_{n=1}^{\infty}
\frac{z^{2n}}{n}(4^n-1)\vert n\vert_3.
\end{eqnarray*}
Following Lemma 4.1 in \cite{ESW} we shall show that $\xi_1(z)$ has infinitely many
logarithmic singularities on the circle $|z|=\frac{1}{2}$,
each of which corresponds to a zero of the Reidemeister zeta function $R_{\phi}(z)$.

Write $3^a\Divides n$ to mean that $3^a\divides n$ but
$3^{a+1}\notdivides n$. Notice that $3^a\Divides n$ if and only if
$\vert n\vert_3=3^{-a}$. Then $\xi_1$ may be split up according to
the size of $\vert n\vert_3$ as
\begin{eqnarray*}
\xi_1(z)&=&\sum_{j=0}^{\infty}\frac{1}{3^j}\sum_{3^j\Vert n}
\frac{z^{2n}}{n}(4^n-1)=\sum_{j=0}^{\infty}\frac{1}{3^j}\eta_j^{(4)}(z),
\end{eqnarray*}
where
$
\eta_j^{(a)}(z)=\sum_{3^j\Divides n}
\frac{z^{2n}}{n}(a^n-1).
$
Then
\begin{eqnarray*}
\eta_0^{(a)}(z)&=&\sum_{3^0\Divides n}\frac{z^{2n}}{n}(a^n-1)=\sum_{n=1}^{\infty}\frac{z^{2n}}{n}(a^n-1)
-\sum_{n=1}^{\infty}\frac{z^{6n}}{3n}(a^{3n}-1)\\
&=&\log\left(\frac{1-z^2}{1-az^2}\right)-
\frac{1}{3}\log\left(\frac{1-z^6}{1-a^3z^6}\right){\!\!},
\end{eqnarray*}
\begin{equation*}
\eta_1^{(4)}(z)=\sum_{3^1\Divides n}\frac{z^{2n}}{n}(4^n-1)
=\sum_{3^0\Divides n}\frac{z^{6n}}{3n}(4^{3n}-1)
=\textstyle\frac{1}{3}\eta_0^{(4^3)}(z^3),
\end{equation*}
\begin{equation*}
\eta_2^{(4)}(z)=\textstyle\frac{1}{9}\eta_0^{(4^9)}(z^9),
\end{equation*}
and so on.
Thus
\begin{equation*}
\xi_1(z)=\log\left(\frac{1-z^2}{1-(2z)^2}\right)+
2\sum_{j=1}^{\infty}\frac{1}{9^j}\log\left(\frac{1-(2z)^{2\times3^j}}{1-z^{2\times3^j}}
\right){\!\!},
\end{equation*}
so for the Reidemeister zeta function we have
\begin{equation*}
\left\vert R_{\phi}(z)\right\vert=
\left\vert\frac{1-z}{1-2z}\right\vert\cdot
\left\vert\frac{1-(2z)^{2\vphantom{3^j}}}{1-z^2}\right\vert^{1/2}\!\!\!\!\cdot
\left\vert\frac{1-z^{2\vphantom{3^j}}}{1-(2z)^{2}}\right\vert^{1/6}\!\!\!\!\cdot
\prod_{j=1}^{\infty}\left\vert
\frac{1-(2z)^{2\times3^j}}{1-z^{2\times3^j}}\right\vert^{1/3\times9^j}.
\end{equation*}

It follows that the series defining the Reidemeister zeta function $R_{\phi}(z)$ has a zero at all
points of the form $\frac{1}{2}e^{2\pi ij/3^r}$, $r\ge1$ so $\vert
z\vert=\frac{1}{2}$ is a natural boundary for the Reidemeister zeta function $R_{\phi}(z)$.
\end{proof}

\end{document}